\documentclass[12pt]{amsart}
\usepackage{fullpage,graphicx,amsfonts,amssymb,amsmath,amsthm,mathrsfs,mathtools}
\usepackage{longtable}
\usepackage[all]{xy}
\usepackage{color} 
\usepackage{hyperref}

\theoremstyle{plain} 
\newtheorem{theorem}    {Theorem}

\newtheorem{lemma}      [theorem]{Lemma}
\newtheorem{corollary}  [theorem]{Corollary}
\newtheorem{proposition}[theorem]{Proposition}

\newtheorem{conjecture} [theorem]{Conjecture}

\theoremstyle{definition}
\newtheorem{definition} [theorem]{Definition}

\theoremstyle{remark}
\newtheorem{remark}              {Remark}

\newcommand{\Mod}[1]{\ (\text{mod}\ #1)}

\newcommand{\odd}{\operatorname{odd}}
\newcommand{\BC}{\operatorname{BC}}
\newcommand{\Tr}{\operatorname{Tr}}
\newcommand{\SL}{\operatorname{SL}}

\newcommand{\legendre}[2]{\genfrac{(}{)}{1pt}{0}{#1}{#2}}

\newcommand{\ps}[1]{[\![#1]\!]}

\begin{document}

\title{Newton Polygons of Hecke Operators}

\author{Liubomir Chiriac \and Andrei Jorza}


\address{Portland State University, Fariborz
  Maseeh Department of Mathematics and Statistics, Portland, OR
  97201}
\email{chiriac@pdx.edu}
\address{University of Notre Dame, 275 Hurley Hall, Notre Dame, IN 46556}\email{ajorza@nd.edu}


\begin{abstract}
In this computational paper we verify a truncated version of the Buzzard-Calegari conjecture on the Newton polygon of the Hecke operator $T_2$ for all large enough weights. We first develop a formula for computing $p$-adic valuations of exponential sums, which we then implement to compute $2$-adic valuations of traces of Hecke operators acting on spaces of cusp forms. Finally, we verify that if Newton polygon of the Buzzard-Calegari polynomial has a vertex at $n\leq 15$, then it agrees with the Newton polygon of $T_2$ up to $n$.




\end{abstract}


\maketitle

\section{Introduction}

Let $2k\geq 12 $ be an even number and let $S_{2k}$ be the finite-dimensional $\mathbb{C}$-vector space of cusp forms of weight $2k$ on $\SL_2(\mathbb{Z})$. For a prime number $p$ and $f\in S_{2k}$, the action of the Hecke operator $T_p$ on $f$ is given by
$$(T_p f)(z)=\frac{1}{p}\sum_{r=0}^{p-1}f\left(\frac{z+r}{p} \right)+p^{2k-1}f(pz).$$

Motivated by a question of Serre, Hatada
\cite{hatada:eigenvalues} obtained several congruences modulo powers of 2 satisfied
by the eigenvalues $a_2$ of $T_2$, later improved by Emerton \cite{emerton:thesis}. More precisely, among the normalized eigenforms in $S_{2k}$ the lowest 2-adic valuation of $a_2$ is 3 (with multiplicity 1) if $k\equiv 0\pmod{2}$, 4 (with multiplicity 1) if $k\equiv 1\pmod{4}$, 5 (with multiplicity 1) if $k\equiv 3\pmod{8}$, and 6 (with multiplicity 2) if $k\equiv 7\pmod{8}$.

Hatada's congruences represent some of the first results
concerning the 2-adic valuations of the eigenvalues of $T_2$,
refered to as $T_2$-slopes. The list of slopes is
determined by the 2-adic Newton polygon of the characteristic polynomial
$$P_{T_2}(X):=\det(1-T_2X\mid S_{2k} )\in \mathbb{Z}[X],$$
which is defined as the convex hull of the set of points $(i,v_2(b_i))$, where $v_2$ denotes the 2-adic valuation, and $b_i$ is the coefficient of $X^i$ in $P_{T_2}(X)$.

In \cite{buzzard-calegari:slopes} Buzzard and Calegari conjectured:

\begin{conjecture}[Buzzard-Calegari]\label{c:bc}
If $m=\dim S_{2k}$ then the polynomial 
$$P_{BC}(X):=1+\sum_{n=1}^m X^n \prod_{j=1}^n \frac{2^{2j}(2k-8j)!(2k-8j-3)!(2k-12j-2)}{(2k-12j)!(2k-6j-1)!}$$ 
has the same 2-adic Newton polygon as the characteristic polynomial $P_{T_2}(X)$. 
\end{conjecture}

This conjecture naturally belongs to the line of questions
raised by Buzzard \cite{buzzard:questions}, in an effort to
better understand the geometry of Coleman-Mazur eigencurves. Indeed, one can reformulate Conjecture~\ref{c:bc} as a statement about the Newton polygon of the characteristic polynomial of $U_2$ acting on the space of $2^{-1/2}$-overconvergent $2$-adic cusp forms of tame level one and nonpositive integer weight. More general conjectures have been formulated for other primes in \cite{liu-wan-xiao} and \cite{bergdall-pollack:ghost1} in connection with Coleman's spectral halo conjecture. 

From a different perspective, the polynomial $P_{T_2}(X)$ is particularly relevant in the context of Maeda's conjecture \cite{hida-maeda}. For example, it has been checked numerically in \cite{ghitza-mcandrew}  that $P_{T_2}(X)$ is irreducible over $\mathbb{Q}$ and its Galois group is the full symmetric group of degree $\dim S_{2k}$ for all $2k\leq 14000$. Furthermore, the irreducibility of $P_{T_2}(X)$ in all weights would imply the same property for the characteristic polynomial of $T_p$ for a density one set of primes $p$ (see \cite{baba-murty}).

The main result of this paper is a computational verification of a truncated version of the Buzzard-Calegari conjecture. In what follows, we denote by $N(P)$ the Newton polygon of a polynomial $P(X)$. If $N$ is a Newton polygon by the truncation $N^{\leq m}$ at $m$ we mean the portion of $N$ in the region $0\leq n\leq m$. We remark that $N(P^{\deg\leq m})\geq N(P)^{\leq m}$, but the two polygons need not be the same. In Theorem \ref{c:main} we prove the following:

\begin{theorem} \label{MainTh}
For $k\gg 0$ the 2-adic Newton polygon $N_{\BC}$ of the Buzzard-Calegari polynomial $P_{\BC}$ has some vertex $n$ 
in the interval $[7,15]$. For any such vertex: \[N_{\BC}^{\leq n}=N(P_{T_2}^{\deg\leq n}).\] Moreover, $N(P_{T_2}^{\deg\leq 15})\geq N_{\BC}$.
\end{theorem}

Theorem \ref{MainTh} relies on a general method for computing
$p$-adic valuations of exponential sums (see \S\ref{exp sums},
particularly Theorem~\ref{t:vp seq}). Using the Eichler-Selberg
trace formula we write the coefficients of $P_{T_2}$ as
exponential sums whose 2-adic valuations we can express in
closed form, which we made explicit for $P_{T_2}^{\deg \leq 15}$
in Sage \cite{sage}. Finally, using these explicit formulas we
verify computationally Theorem \ref{MainTh}. 

It is important to emphasize that Theorem \ref{MainTh} no longer
holds if we replace the claim $N_{\BC}^{\leq
  n}=N(P_{T_2}^{\deg\leq n})$ for a vertex $n$ of $N_{\BC}$,
with $N_{\BC}^{\leq 15}=N(P_{T_2}^{\deg\leq 15})$. For instance,
if $k\equiv 131126\pmod{2^{17}}$ then $N(P_{T_2}^{\deg\leq 15})$
has a vertex at 14, whereas $N_{\BC}^{\leq 15}$ does not. 

Nevertheless, we show an effective relationship between the Newton polygons of $P_{T_2}$ and its truncation: for each $n\geq 1$ there exists an $m_{n,k}$ given by an explicit formula in terms of $n$ and $k$ with the property that $n$ is a vertex of $N(P_{T_2})$ if and only if it is a vertex of $N(P_{T_2}^{\deg\leq m_{n,k}})$ (see Lemma \ref{l:vertex T2}). In principle, with sufficient computing power one would be able to verify that up to a vertex $n\leq 15$, the Newton polygons of the Buzzard-Calegari conjecture coincide, i.e., that $N_{\BC}^{\leq n}=N(P_{T_2})^{\leq n}$. Our computations suffice to show that $N_{\BC}^{\leq 4} = N(P_{T_2})^{\leq 4}$ for all sufficiently large $k$ and, e.g., that $N_{\BC}^{\leq 13} = N(P_{T_2})^{\leq 13}$ for $k\equiv 34\pmod{2^{11}}$.

Finally, we record one consequence of Theorem \ref{MainTh}  to Hatada's congruences: 

\begin{corollary}  \label{cor}
Among the normalized eigenforms in $S_{2k}$, for $k\gg 0$, the lowest two 2-adic valuations of $a_2$ are:
\begin{center}
\begin{tabular}{ll || ll || ll}
$3, 7$ & if $k\equiv 0\pmod{4}$ &$4, 9$ & if $k\equiv 5\pmod{16}$ &$5, 8$ & if $k\equiv 3\pmod{16}$ \\
$3, 8$ & if $k\equiv 6\pmod{8}$ &$4, 10$ & if $k\equiv 29\pmod{32}$ &$5, 9$ & if $k\equiv 27\pmod{32}$ \\
$3, 9$ & if $k\equiv 2\pmod{8}$ &$4, 11$ & if $k\equiv 45\pmod{64}$ &$5, 10$ & if $k\equiv 11\pmod{32}$\\
$4, 8$ & if $k\equiv 1\pmod{8}$ &$4, 12$ & if $k\equiv 13\pmod{64}$ &$6, 6$ & if $k\equiv 7\pmod{8}$.\\
\end{tabular}
\end{center}
Moreover, a consequence of $N_{\BC}^{\leq 4}=N(P_{T_2})^{\leq 4}$ is that one can list exhaustively the lowest four 2-adic valuations of $a_2$.
\end{corollary}

The article is organized as follows: in \S\ref{exp sums} we give a closed form expression for $p$-adic valuations of finite exponential sums; in \S\ref{s:hecke} we apply our previous findings to Hecke traces (Corollary~\ref{c:wedge n}); in \S\ref{refine} we discuss certain computational improvements obtained from sums involving Hurwitz class numbers; finally, in \S\ref{s:bc}, we explain our computational verification of Theorem \ref{MainTh}. The explicit terms that appear in the 2-adic valuations of the traces of $T_2$ are tabulated in the Appendix. 

\bigskip

\textit{Acknowledgements}. We are grateful to Damien Roy for many useful suggestions, and especially for showing us the proof of Proposition~\ref{c:super}, to the anonymous referee for encouraging us to work out the relationship between the Newton polygons of $P_{T_2}$ and its truncation, and to John Bergdall for helpful conversations.

\section{Valuations of exponential sequences} \label{exp sums}

 In this section we describe a general method for computing valuations of finite exponential sums. When applying the recipe below to the exponential sums coming from the trace formula we encounter a technical complication reminiscent of Schanuel's conjecture, but which can be circumvented computationally, if not theoretically.

Let $K/\mathbb{Q}_p$ be a finite extension with ring of integers
$\mathcal{O}$, uniformizer $\varpi$ and residue field
$\mathbb{F}$. We denote $v_K$ the valuation on $K$ with $v_K(\varpi)=1$ and $\omega$ any section of $\mathcal{O}\to \mathbb{F}$, e.g., the Teichmuller lift. In this section we compute the valuations of certain functions on $\mathcal{O}$.

We begin with a lemma describing an explicit algorithm for computing valuations of polynomials.

\begin{lemma}\label{l:vp poly}
Let $P(x)\in K[x]$ be a monic irreducible polynomial
of degree $d\geq 2$ and slope
$\lambda\in\frac{1}{d}\mathbb{Z}$. Suppose that $k\in \mathcal{O}$.
\begin{enumerate}
\item If $\lambda<0$ then $v_K(P(k))=v_K(P(0))$.
\item If $\lambda\geq 0$ then there exists a (possibly empty) sequence $u_0,u_1,\ldots, u_n\in
\mathbb{F}$ and $\delta\in \mathbb{Z}$ such that
\[v_K(P(k)) = \min(d (\lambda+n+1)+\delta , dv_K(k-\sum\limits_{i=0}^n \varpi^{\lambda+i}\omega(u_i))).\]
\end{enumerate}

\end{lemma}

\begin{proof}
The first part is straightforward. In the second part, if $\lambda$ is not integral we may choose $\delta=0$ and the empty sequence. If $\lambda\in \mathbb{Z}_{\geq 0}$ we remark that the polynomial $\varpi^{-d \lambda}P(\varpi^\lambda x)$ is irreducible with slope 0 and therefore it suffices to treat the case $\lambda=0$.

We construct a sequence of
 integral polynomials $(P_n(X))_{n\geq 0}$ and a sequence
$(u_n)_{n\geq 0}$ as follows. Let $P_0(x)=P(x)$. Suppose we have
already constructed monic irreducible polynomials
$P_0(x),\ldots, P_n(x)$ with slopes
$\ell_0,\ldots,\ell_n\in\mathbb{Z}_{\geq 0}$. If $\ell_n<0$ or
$\ell_n\notin \mathbb{Z}$ or $P_n\!\!\mod\varpi$ has no root in $\mathbb{F}$ the sequence terminates at $P_n(x)$. Otherwise, $\ell_n\in \mathbb{Z}_{\geq 0}$ and $P_n(x)\!\!\mod\varpi$ has a single root in $\mathbb{F}$, which we denote $u_n$. Define $P_{n+1}(x)=\varpi^{-d}P_n(\varpi x+\omega(u_n)).$ The polynomial $P_{n+1}(x)$ is also monic irreducible and therefore isoclinic with slope $\ell_{n+1}$, and we proceed as above.

First, we show that the process always terminates. Otherwise,
$\ell_{n}\in \mathbb{Z}_{\geq 0}$ and
$P_{n+1}(x)=\varpi^{-d}P_n(\varpi x+\omega(u_n))$ for all $n$. In this case,
\[\varpi^{nd}P_{n}(0)=P \left(\sum_{i=0}^{n-1}\varpi^i \omega(u_i) \right)\]
for all $n$. We conclude that $\sum\limits_{i=0}^\infty \varpi^i \omega(u_i)\in
K$ is a root of $P(x)$, contradicting its
irreducibility.

Suppose now that the sequence of polynomials is $P_0,\ldots,
P_{n+1}$ with associated sequence $u_0,\ldots, u_n\in
\mathbb{F}$, in which case either $P_{n+1}(x)$ has slope $\frac{\delta}{d}$ not in
$\mathbb{Z}_{\geq 0}$ or $P_{n+1}$ has no root in $\mathbb{F}$, in which case we set $\delta=0$. We will prove the
desired formula by induction on $n$. The base case is
$n=-1$. Since $P(x)$ has nonnegative integral slope we deduce
that $P(x)$ has no root in $\mathbb{F}$, $\lambda=0$ and $\delta=0$ by
definition and therefore
$v_K(P(k))=0=\min(d(\lambda+n+1)+\delta,dv_K(k))$.

For the induction step, note that for the polynomial
$P_1(x)$ the sequence is $P_1,\ldots, P_{n+1}$. First, if $k\not\equiv u_0\pmod{\varpi}$ then $v_K(P(k))=0$. Otherwise, if $k\equiv u_0\pmod{\varpi}$ the inductive
hypothesis implies that
\begin{align*}
  v_K(P(k))&=d+v_K(P_1((k-\omega(u_0))/\varpi))\\
&=d+\min(d(\lambda+n)+\delta, d v_K((k-\omega(u_0))/\varpi-\sum_{i=0}^{n-1}\varpi^{i} \omega(u_{i+1}))\\
&=\min(d(\lambda+n+1)+\delta, d v_K(k-\sum_{i=0}^{n}\varpi^{i} \omega(u_{i})).
\end{align*}

\end{proof}

We turn to the main results of this section. Our aim is to
compute, for each positive integer $k$, the valuation of a finite
exponential sum of the form:
\[f(k)=\sum_{n=1}^m a_n b_n^k,\]
where $a_1,\ldots, a_m\in K$ and $b_1,\ldots, b_m\in
\mathcal{O}$ are nonzero. 
\begin{proposition}\label{p:vp seq}
Suppose $a_1,\ldots, a_m\in K^\times$ and $b_1,\ldots, b_m\in \mathcal{O}^\times$. There exists an integer $D$ and for
each class $r\pmod{D}$ there exist $\lambda_r\in \mathbb{Z}$
 and possibly empty collections of $\Omega_{r,i}, u_{r,j}\in \mathcal{O}$, $n_{r,j}\in \mathbb{Z}_{\geq 0}$, and $d_{r,j}\in \mathbb{Z}_{\geq 2}$ such that
\[v_K(f(k)) = \lambda_r+\sum_i v_K(k-\Omega_{r,i})+\sum_j\min(n_{r,j}, d_{r,j} v_K(k-u_{r,j})),\]
for each $k\equiv r\pmod{D}$.

Moreover, all the constants above are effective and can be computed up to arbitrary precision in polynomial time.
\end{proposition}
\begin{proof}
As usual we denote
$q=|\mathcal{O}/\varpi \mathcal{O}|$ and $e=e_{K/\mathbb{Q}_p}$ for the ramification index.

Let $D$ be the smallest positive integer
such that $b_n^D\equiv 1\pmod{\varpi^\ell}$ for each $n$, where
$\ell=e+1$, in which case
$D\mid q^{\ell-1}(q-1)$. We denote $c_n\in \mathcal{O}$ such that $b_n^D=1+\varpi^\ell c_n$. If $k=Dx+r$ for some integer $x$ then
\[f(k)=\sum_{n=1}^m a_nb_n^r (b_n^D)^x=\sum_{n=1}^m a_nb_n^r
(1+\varpi^\ell c_n)^x.\]
Recall that $\log_p(1+x)$ converges absolutely when $v_p(x)>0$
and $\exp(x)$ converges absolutely when
$v_p(x)>\frac{1}{p-1}$. In the latter case, $\log_p$ and $\exp$
are mutual inverses. Since $v_p(\varpi^\ell)>1$ it follows that $v_p(\log_p(1+\varpi^\ell c_n))>1\geq \frac{1}{p-1}$ and therefore
\[(1+\varpi^\ell c_n)^x = \exp(x\log_p(1+\varpi^\ell c_n)),\]
is a $p$-adic analytic function in $x$. As a result, we see that the $p$-adic analytic function
\[f_r(x)=\sum_{n=1}^m a_nb_n^r \exp(x\log_p(1+\varpi^\ell
c_n))\]
satisfies $f(Dx+r)=f_r(x)$ for each integer $x$.

Examining the Taylor expansion of $f_r(x)$ we see that
\[f_r(x)=\sum_{N=0}^\infty \frac{x^N}{N!}\sum_{n=1}^m a_n b_n^r
\left(\log_p(1+\varpi^\ell c_n)\right)^N\in K\otimes
\mathcal{O}\ps{x},\]
since $v_p(\left(\log_p(1+\varpi^\ell c_n)\right)^N/N!)=N(v_p(\log_p(1+\varpi^\ell c_n))-1)\to \infty$ and so $f_r(x)$ converges on all of $\mathcal{O}$, since its coefficients go to 0. 
Let $d$ be the $x$-coordinate of the lowest vertex of the Newton Polygon of $f_r(x)$, and $\alpha_r$ be the coefficient of $x^d$. There
exists Newton slope factorization for $f_r(x)$, i.e., a monic polynomial
$P_r(x)\in \mathcal{O}[x]$ of degree $d$ and a power series $u_r(x)\in
1+\varpi\mathcal{O}\ps{x}$ such that $f_r(x)=\alpha_r P_r(x)u_r(x)$,
the Newton Polygon of $f_r(x)$ being the concatenation of the
Newton Polygons of $\alpha_r P_r(x)$ and $\alpha_r u_r(x)$.

For all $x\in \mathcal{O}$ we see that $v_K(u_r(x))=0$ as $u_r(x)$ converges on $\mathcal{O}$. This implies that 
\[v_K(f_r(x))=v_K(\alpha_r)+v_K(P_r(x)).\]
Factoring $P_r(x) = \prod_i (x-\omega_i) \prod R_j(x)$ where each $R_j(x)$ is irreducible of degree $d_j\geq 2$ and applying Lemma \ref{l:vp poly} immediately implies the desired formula for $v_K(f(k))$ for each $k\equiv r\pmod{D}$.

Suppose we want to compute the constants up to precision $M$. First note that there is an explicit integer $D$, computable in polynomial time, such that up to precision $M$ the power series $f_r(x)$ agrees with the truncation of $f_r(x)$ in degree $D$. Indeed, it suffices that $D(v_p(\log(1+\varpi^\ell c_n))-1)>M$ for each $n$. Then computing $f_r(x)$ requires $O(M^2(M+D)m)$ operations, the factor $M^2$ being the trivial upper bound for the running time of multiplication of long numbers. Thereafter, the effectiveness of computing the constants in the formula is equivalent to the effectiveness of Newton slope factorizations and polynomial factorizations in $K[x]$, both of which are polynomial time algorithms by $p$-adic approximation and Hensel's lemma. 

\end{proof}

Proposition \ref{p:vp seq} definitively answers the question of
what $v_p(f(k))$ is under the assumption that each exponential
base $b_i$ is a $p$-adic unit. Let us examine the general
situation. Suppose $a_1,\ldots, a_m\in K^\times$ and
$b_1,\ldots, b_m\in \mathcal{O}-\{0\}$, and let
$f(k)=\sum a_i b_i^k$. Collecting exponential bases by valuation
there exist $\mu_1<\ldots<\mu_s\in \mathbb{Z}_{\geq 0}$ and
exponential functions $f_1,\ldots, f_s$ satisfying the
hypotheses of Proposition \ref{p:vp seq} such that
\[f(k)=\sum_{i=1}^s \varpi^{\mu_i k} f_i(k).\]
Taking $D$ as the least common multiple of the moduli for
$f_1,\ldots,f_s$ we deduce that for each $i$ and $r\mod D$ there
exist constants $\lambda_{r,i}, \Omega_{r,i,j}, u_{r,i,j},
n_{r,i,j}, d_{r,i,j}$ such that
\[v_K(f_i(r+Dk))=\lambda_{r,i}+\sum_j
v_K(k-\Omega_{r,i,j})+\sum_j \min \left( n_{r,i,j}, d_{r,i,j}
v_K(k-u_{r,i,j}) \right).\]
In this nonarchimedean setting whether we can compute the
valuation of $f(k)$ depends on whether the valuations of
$\varpi^{\mu_i k} f_i(k)$ are distinct. At first glance it
seems that the linear term $\mu_i k$ dominates this valuation,
but this need not be the case.

\bigskip

\begin{definition}
We say that $\Omega\in \mathcal{O}$ is \emph{superconvergent} if $$\displaystyle \limsup \frac{v_K(k-\Omega)}{k} >0.$$ For example, the superconvergent elements of $\mathbb{Z}_2$ are precisely those with $2$-adic expansion $\displaystyle \sum_{n=0}^\infty 2^{\ell(n)}$, where $\limsup \ell(n+1)/2^{-\ell(n)}>0.$
\end{definition}

\bigskip

If none of the $\Omega_{r,i,j}$ are superconvergent then indeed
the linear term $\mu_i k$ dominates the valuation of
$\varpi^{\mu_i k} f_i(k)$. However, this is not always the
case. For example, suppose $\Omega$ is superconvergent. Then for
some base $b\in \mathcal{O}^\times$ with $v_K(b-1)$ sufficiently
large the exponential function $f(k)=-b^\Omega \cdot 1^k +
1\cdot b^k$ has root $\Omega$ which then necessarily appears in
the expression for $v_K(f(k))$.

We are grateful to Damien Roy for kindly providing us with the proof of the following proposition, which shows that superconvergent roots only occur in a local setting.

\begin{proposition}\label{c:super}
Suppose $F$ is a number field with ring of integers
$\mathcal{O}$, and let $a_1,\ldots, a_m\in F^\times$. Let
$b_1,\ldots, b_m \in \mathcal{O}-0$ such that no ratio
$b_i/b_j$ is  a root of unity for $i\neq j$. Then for each
embedding $F \hookrightarrow \overline{\mathbb{Q}}_p$, sending
each $b_j$ to a $p$-adic unit,
the exponential function $f(k)=\sum\limits_{i=1}^m a_i b_i^k$ does not have a superconvergent root.
\end{proposition}

\begin{remark}
This result can be reformulated to say that if $\Omega$ is superconvergent and $b_1,\ldots, b_m\in \overline{\mathbb{Q} }$ then the elements $b_1^{\Omega}, \ldots, b_m^{\Omega}$ (interpreted as elements in $\mathbb{C}_p$) are linearly independent over $\overline{\mathbb{Q}}$. Under this formulation Proposition \ref{c:super} is reminiscent of, albeit in no way related to, Schanuel's conjecture.
\end{remark}

\begin{proof}
We denote by $w$ the place of $F$ corresponding to the embedding
into $\overline{\mathbb{Q}}_p$. As explained in the proof of
Proposition \ref{p:vp seq} it suffices to assume that each $b_j$
satisfies $|b_j-1|_w<1$, since a suitable integral power of
$b_j$ satisfies this condition. Moreover, to simplify inequalities, we may assume that
$a_1,\ldots, a_m\in \mathcal{O}$.

The result then follows as an application of
Schlickewei's extension of Schmidt's subspace theorem
\cite{schl}. Let $S$ be a finite set of places that
includes the archimedean places of $F$ and the place $w$ such
that each $b_j$ is an $S$-unit in $\mathcal{O}$. Furthermore, we consider the linear form $L_{w,1}(X_1,\ldots,X_m)=a_1X_1+\cdots+a_mX_m$ and for each place $v\in S$ and index $1\leq i\leq m$ we denote $L_{v,i}(X_1,\ldots,X_m)=X_i$, if $(v,i)\neq (w,1)$. Schlickewei's result implies that for each $\delta>0$ the non-zero solutions $(x_1,\dots,x_m)\in \mathcal{O}^m$ of the inequality 
	\begin{equation}\label{schl}
	\prod_ {v\in S} \prod_{i=1}^m |L_{v,i} (x_1,\dots, x_m)|_v \leq 
		\left( \max_{\substack{ 1\leq i \leq m \\ \sigma:F \hookrightarrow{} \mathbb{C}}} |\sigma(x_i)|
		\right)^{-\delta}	
    \end{equation}
are contained in a union of finitely many proper subspaces of $F^m$, where $|\cdot|_v$ is the usual norm on $F_v$, normalized to give the value $q_v^{-1}$ on uniformizers, where $q_v$ is the cardinality of the residue field.

Let us reinterpret \eqref{schl} in the case when
$(x_1,\dots,x_m)=(b_1^k,\dots,b_m^k)$ for a positive integer
$k$. The assumption that each $b_j$ is an $S$-unit implies that
$\prod\limits_{v\in S}|b_j|_v=1$, as the adelic norm is 1 on
$F^\times$ and $|b_1|_w=1$ by hypothesis. Therefore the left
hand side of \eqref{schl} is
$|a_1b_1^k+\cdots+a_mb_m^k|_w$. Since not all $b_i$ are roots of
unity (this follows from hypothesis if $m\geq 2$, and the
statement is trivially true when $m=1$) it follows that
$C=\displaystyle \max_{\substack{ 1\leq i \leq m \\ \sigma:F
    \hookrightarrow{} \mathbb{C}}} |\sigma(b_i)|>1$. Thus inequality
\eqref{schl} becomes
\[|a_1b_1^k+\cdots+a_mb_m^k|_w<C^{-k \delta}.\]

Suppose now that $f(k)$ has a superconvergent root $\Omega$. By
definition there exists $\varepsilon>0$ and a sequence of
positive integers $(n_k)$ such that $v_p(\Omega-n_k)>\varepsilon
n_k$ for $k\gg 0$. Note that for each $j$ and $k$,
$v_p(b_j^{\Omega-n_k}-1)=v_p(b_j-1)+v_p(\Omega-n_k)>\varepsilon
n_k$ and so
\[|\sum_{j=1}^m a_j b_j^{n_k}|_w = |\sum_{j=1}^m a_j b_j^{n_k}
(1-b_j^{\Omega-n_k})|_w<q_w^{- \varepsilon n_k}.\]
Choosing $\delta>0$ such that $C^{-\delta}=q_w^{-\varepsilon}$
implies that the tuples $(b_1^{n_k},\ldots, b_m^{n_k})$ lies in
finitely many proper subspaces of $F^m$, for $k\gg 0$.

Suppose $c_1x_1+\cdots c_m x_m=0$ is one of these proper subspaces. 
The Skolem-Mahler-Lech Theorem  \cite[Theorem 2.1]{Skolem}
implies that if $(b_1^{n_k},\ldots, b_m^{n_k})$ lies on this subspace then $n_k$ belongs to a finite set or a finite union of arithmetic progressions. The assumption that $b_i/b_j$ is not a root of unity implies that $\sum c_j b_j^r=0$ cannot be satisfied for $r$ in arithmetic progressions and therefore the equation holds only for finitely many $n_k$, yielding the desired contradiction.
\end{proof}

We now turn to the general setting of exponential sums. Suppose $F$ is a number field with ring of integers
$\mathcal{O}$ and $K/\mathbb{Q}_p$ is a finite extension containing $F$.
\begin{theorem}\label{t:vp seq}
Suppose $a_1,\ldots, a_m\in F^\times$ and
$b_1,\ldots, b_m\in \mathcal{O}-\{0\}$, such that $b_i/b_j$ is not a root of unity for $i\neq j$. Let $\mu=\min v_K(b_i)$. There there exists an integer $D$ and for
each class $r\pmod{D}$ there exist $\lambda_r\in \mathbb{Z}$,
 and possibly empty collections of $\Omega_{r,i}, u_{r,j}\in \mathcal{O}$, $n_{r,j}\in \mathbb{Z}_{\geq 0}$, and $d_{r,j}\in \mathbb{Z}_{\geq 2}$ such that
 \[
 	v_K(f(r+Dk)) = \mu k + \lambda_r+\sum_i v_K(k-\Omega_{r,i})+\sum_j\min(n_{r,j}, d_{r,j} v_K(k-u_{r,j})).
 \]
for each $k$ large enough. Moreover, all the constants above are effectively computable to arbitrary precision in polynomial time.
 		
\end{theorem}
\begin{proof}
As mentioned above, if we order $\mu_1<\ldots<\mu_s$ the distinct valuations of the exponential bases then $v_K(f(k))=v_K(\sum \varpi^{\mu_i k}f_i(k))$
where for $k\equiv r\pmod{D}$:
\[v_i(k)=v_K(\varpi^{\mu_i k}f_i(k))=\mu_{r,i}k + \lambda_{r,i}+\sum_j
v_K(k-\Omega_{r,i,j}) + \sum_j\min(n_{r,i,j},
d_{r,i,j}v_K(k-u_{r,i,j})).\]
We choose $\mu=\mu_1$, $\lambda_r=\lambda_{r,1}$, $\Omega_{r,j}=\Omega_{r,1,j}$, etc. 

Using Proposition \ref{c:super} it follows that $v_i(k)-\mu_i
k = O(\sum_j v_K(k-\Omega_{r,i,j})) = o(k)$ and therefore for
$k$ large enough $v_1(k)<v_i(k)$ for all $i>1$. We conclude that
$v_K(f(k))=v_1(k)$ as desired.



\end{proof}

\begin{remark}
Theorem \ref{t:vp seq} yields a computable expression for the valuations of exponential sums. However, we cannot make effective the condition that $k$ be large enough. 
\end{remark}

\section{Traces of Hecke operators}\label{s:hecke}
As discussed in the introduction, Maeda's conjecture is
intimately related to the characteristic polynomials
$P_{T_p}(X)=\det(1-T_pX|S_{2k})$. In this section we turn our
attention to algorithmically computing the Newton polygon of
$P_{T_p}(X)$, which we implement in the last section for $T_2$.

First, we note that if $S_{2k}$ has dimension $m$ then $$P_{T_p}(X)=\sum\limits_{n=0}^m (-1)^n X^n \Tr(\wedge^n T_p|S_{2k}).$$ To compute the trace of $\wedge^n T_p$ we follow the usual method of expressing elementary symmetric polynomials in terms of power sums by means of exponential Bell polynomials. The exponential Bell polynomial $B_n(x_1,\ldots,x_n)\in \mathbb{Z}[x_1,\ldots, x_n]$ has degree $n$ and $p(n)$ monomial terms, where $p(n)$ is the partition function, and we compute:
\[\Tr(\wedge^n T_p|S_{2k}) = \frac{(-1)^n}{n!}B_n(-0!\cdot\Tr(T_{p}|S_{2k}), -1!\Tr(T_{p}^2|S_{2k}),\ldots, -(n-1)!\Tr(T_{p}^n|S_{2k})).\]
Using the definition of $T_{p^n}$ and the Newton identities, we see that:
\[\Tr (T_p^n|S_{2k}) = \sum_{i=0}^{\lfloor
  n/2\rfloor}\left(\binom{n}{i}-\binom{n}{i-1}\right)p^{(k-1)i}\Tr(T_{p^{n-2i}}|S_{2k})=
\Tr(T_{p^n}|S_{2k})+O(p^{k-1}),\]
which implies that $\Tr(\wedge^n T_p|S_{2k}) = \frac{(-1)^n}{n!}B_n(\ldots,-(i-1)!\cdot\Tr(T_{p^i}|S_{2k}),\ldots)+O(p^{k-1})$. Here $A(k)=\mathcal{O}(B(k))$ for two $p$-adic functions $A$ and $B$ means that $|A(k)|_p \ll |B(k)|_p$ for $k\gg 0$.

\bigskip

Finally, we may compute the traces of $T_{p^n}$ using the
Eichler-Selberg trace formula, which expresses the trace of the
Hecke operator in terms of Hurwitz class numbers. Since it
will be necessary later, we recall here the definition of the
Hurwitz class number $H(m)$ (as in \cite[p. 306]{knightly-li:traces}). For a positive integer $m\equiv
0,3\pmod{4}$ we denote $E_m = \mathbb{Q}(\sqrt{-m})$ with
discriminant $-d_m$ in which case we may write $m=a^2d_m$
for a positive integer $a$. We define $\legendre{E_m}{\ell}$ to
be $\legendre{-d_m}{\ell}$ if $\ell$ is odd, and if $\ell=2$ then either 0 if $d_m$ is even or $(-1)^{(d+1)/4}$ otherwise. For an integer $a$ we denote
$$\varphi_m(a)=\displaystyle a\prod_{\ell\mid
  a}\left(1-\legendre{E_m}{\ell}\frac{1}{\ell}\right),$$ in which case the Hurwitz class number is 
  
$$H(m)=\frac{h_{m}}{w}\sum_{a^2\mid m/d_m}\varphi_m(a)$$

where $h_m$ is the class number of $E_m$, and 

$$w=
\begin{cases}
2 \text{ if }m=4, \\
3 \text{ if }m=3, \\
1 \text{ otherwise.}  \\
\end{cases}
$$

In the case of the Hecke operator $T_{p^n}$ the Eichler-Selberg trace formula \cite[p. 370]{knightly-li:traces} implies:
\[\Tr(T_{p^n}|S_{2k})=-\frac{1}{2}\sum_{i=0}^n p^{(2k-1)\min(i,n-i)}-\frac{H(4p^n)(-p^n)^{k-1}}{2}-\sum_{1\leq t<2p^{n/2}}H(4p^n-t^2)\frac{\rho_t^{2k-1}-\overline{\rho}_t^{2k-1}}{\rho_t-\overline{\rho}_t}.\]
Here  $\rho_t,\overline{\rho}_t$ are the roots $\dfrac{t\pm \sqrt{t^2-4p^n}}{2}$ of $X^2-tX+p^n=0$, chosen such that $v_p( \rho_t) \leq v_p( \overline{\rho}_t )$.  

An immediate consequence is that $\Tr(\wedge^n T_p|S_{2k})$ is an exponential sum of the type considered in the previous section and, by Theorem \ref{t:vp seq}, to compute $v_p(\Tr(\wedge^nT_p|S_{2k}))$ it suffices to consider only the exponentials with $p$-adic unit bases.

\begin{lemma} \label{l:hecke}
Let $n$ be a positive integer and $p$ a prime number. Then
\begin{align*}
\Tr(T_{p^n}|S_{2k})&= -1-\sum_{1\leq t<2p^{n/2}, p\nmid
                  t}\frac{H(4p^n-t^2)}{\sqrt{t^2-4p^n}}\rho_t^{2k-1}+O(p^{k})\\
\Tr(T_{2}|S_{2k})&= -1-\frac{1}{\sqrt{-7}}\rho_1^{2k-1} +O(2^{k}).
\end{align*} 
\end{lemma}

\begin{proof}
Note that for every prime $p$: $$\displaystyle
\frac{1}{2}\sum_{i=0}^np^{(k-1)\min(i,n-i)} =
1+O(p^{k}).$$ 

For the first identity it suffices to check that only the
above mentioned roots of $X^2-tX+p^n=0$ remain. If $p$ is odd,
the bound $t<2p^{n/2}$ implies that $v_p(t)<n/2$ and so the
Newton Polygon of this equation consists of two segments with
negative slope, one of which equals $-v_p(t)$. We conclude that
both roots are in $\mathbb{Z}_p$ with valuations $v_p(t)$ and
$n-v_p(t)$. In fact, $\rho_t$ has valuation $v_p(t)$ and
$\overline{\rho}_t$ has valuation $n-v_p(t)$. Since these roots
are raised to the exponent $2k-1$ we may remove all except
$\rho_t$ with $p\nmid t$.

When $p=2$ the equation is $X^2-2^\ell u X+2^{n}=0$, where $t=2^\ell u$ for an odd number $u$. A priori we know that $2\ell <n+2$. If $2\ell < n$ then the Newton Polygon consists again of two segments, with slopes $\ell$ and $n-\ell$. Otherwise it consists of a single segment with slope $n/2$. If $n>1$ then $n-\ell>0$ and so we may eliminate the roots $\overline{\rho}_t$ and keep $\rho_t$ for odd $t$. When $n=1$ the above expression is immediate by inspection.

\end{proof}

Turning back to $f(k)=\Tr(\wedge^n T_p|S_{2k})$ we see that the
exponential sum
\[f(k) = \frac{(-1)^n B_n(0!,1!,\ldots, (n-1)!)}{n!} + \sum a_i b_i^k\]
has the following properties:
\begin{enumerate}
\item $B_n(0!,1!,\ldots, (n-1)!)\neq 0$ (the coefficients are all
nonnegative) and
\item The remaining bases $b_i$ are not roots of unity. Indeed, each base $b_i$ is of the form $b_i = \prod \rho_{j,
  t_j}^{e_j}$ where $\rho_{j, t_j}$ is a root of
$X^2-t_jX+p^j=0$. However, each root $\rho_{j,t_j}$ has norm $p^j$ which means that $b_i$ has non-unit norm.
\end{enumerate}

We now apply the results of the previous section to Hecke traces. In this case we may choose $D=2$ and $r=0$, since level one modular forms have even weights.

\begin{corollary}\label{c:wedge n}
For each $n\geq 1$ there exists an integer $\lambda_n\in \mathbb{Z}$ and possibly empty collections $\Omega_{i}\in \mathbb{Z}_2$ and integers $n_{j}\geq 0, d_{j}\geq 2, u_{j}\geq 0$ such that

\begin{equation}\label{eq:np}v_2(\Tr(\wedge^n T_2|S_{2k}))=\lambda_n +\sum_i
v_2(k-\Omega_{i})+\sum_j\min(n_{j}, d_{j} v_2(k-u_{j})),
\end{equation}
for all large enough $k$. 
\end{corollary}

\begin{proof}
This follows from Theorem \ref{t:vp seq} since the
exponential series for $\Tr(\wedge^n T_2|S_{2k})$ is not the 0 power series, as explained above.

\end{proof}

\section{Computations of valuations of Hecke traces} \label{refine}
From an algorithmic perspective Corollary \ref{c:wedge n} has a
major flaw: while the constants in Theorem \ref{t:vp seq} are computable in polynomial time, the running time is a polynomial in the number of exponentials, which in the case
of traces of Hecke operators is on the order of $O(2^{n/2})$. Note that even if the number of terms in the Eichler-Selberg trace formula were smaller, the exponential Bell polynomials have exponentially many monomials, which means any implementation of Theorem \ref{t:vp seq} for $\Tr(\wedge^nT_2|S_{2k})$ would have to be exponential in $n$.

In this section we explain our implementation of Corollary \ref{c:wedge n} with an eye towards speeding up this exponential time algorithm.
We begin with an observation. Suppose $E_1(k),\ldots, E_n(k)$ are exponential sums over $\mathbb{Z}_2$ and $B(X_1,\ldots, X_n)\in\frac{1}{T} \mathbb{Z}[X_1,\ldots,X_n]$ for some $T\in \mathbb{Z}_2$. If $f_{r,1}(k),\ldots, f_{r,n}(k)$ are the power series associated with the exponential sums $E_1,\ldots, E_n$ for $k\equiv r\pmod{D}$ then $B(f_{r,1}(k),\ldots, f_{r,n}(k))$ is the power series associated with $B(E_1,\ldots, E_n)$. This computation is polynomial in the precision of $\mathbb{Z}_2$, of $\mathbb{Z}_2\ps{k}$, and of the number of monomials in $B$. Moreover, if $g_{r,i}\equiv f_{r,i}\pmod{2^M}$ then $$B(g_{r,1},\ldots, g_{r,n})\equiv B(f_{r,1},\ldots, f_{r,n})\pmod{2^{M-v_2(T)}}$$ which means that instead of approximating the exponential sum $B(E_1,\ldots, E_n)$ with power series to precision $N$ it suffices to approximate each $E_i$ to precision $M+v_2(T)$.

Let us turn to making more efficient the computations of power
series attached to the exponential sums in
$\Tr(T_{2^n}|S_{2k})$, which can then be used to compute the
power series attached to $\Tr(\wedge^n T_2|S_{2k})$ as in the
previous paragraph. A priori the computation involves $2^{n/2}$
exponentials:
\[\Tr(T_{2^n}|S_{2k})=-1-\sum_{1\leq t<2^{n/2+1},2\nmid t}\frac{H(2^{n+2}-t^2)}{\sqrt{2^{n+2}-t^2}}\rho_t^{2k-1}+\mathcal{O}(p^k).\]
We are able to make this computation more efficient by
controlling the precision of computations. Suppose we seek the
power series $f(k)$ attached to the exponential sum above up to
precision $M$. We compute
\begin{align*}
\Tr(T_{2^n}|S_{2k})&=-1-\sum_{0\leq u<2^{n/2}}\sum_{m=0}^\infty
  \frac{k^m}{m!}\frac{H(2^{n+2}-(1+2u)^2)}{\rho_{1+2u}\sqrt{2^{n+2}-(1+2u)^2}}\left(\log_2(\rho_{1+2u}^2)\right)^m\\
  &=-1-\sum_{0\leq u<2^{n/2}}\sum_{m=0}^\infty
  \frac{k^m}{m!}H(2^{n+2}-(1+2u)^2)\sum_{r=0}^\infty a_{n,m,r}u^r
\end{align*}
where the coefficients $a_{n,m,r}$ can be computed from the
Taylor expansion around 0. They satisfy $v_2(a_{n,m,r})\geq
2m$ for all $r$ and, for any $\alpha<1$, $v_2(a_{n,m,r})\geq
2m+\alpha r$ for $r\geq r_\alpha = \lceil \frac{4}{1-\alpha}\rceil + \frac{2 \alpha}{1-\alpha}\lceil \frac{2}{1-\alpha}\rceil$. Indeed, this can be
seen as follows. As a power series in $u$ we have
$\log_2(\rho_{1+2u}^2)\in 2^2 \mathbb{Z}_2\ps{u}$. Factoring
$2^2$ we obtain the term $2m$ in the lower bound on
valuation. Let us turn to the growth of valuations of the Taylor coefficients of
$q(u)=(2^{-2}\log_2(\rho_{1+2u}^2))^m/(\rho_{1+2u}\sqrt{2^{n+2}-(1+2u)^2})$, or equivalently to the sizes of slopes of the Newton polygon of $q(u)$. Any root of $q(u)$ would satisfy $\rho_{1+2u}=\zeta$, where $\zeta$ is a root
of unity of order $2^b$ for some $b$. But then $1+2u=\zeta +
2^{n+2}\zeta^{-1}$ and so $u=(\zeta-1)/2 + 2^{n+1}\zeta^{-1}$
which has valuation $1/2^{b-1}-1$. We conclude that the number of slopes of the Newton polygon of $q(u)$ which are $\leq \lambda$ for some $\lambda <1$ is at most $\lceil\frac{2}{1-\lambda}\rceil$, since they correspond to roots of unity with $1-1/2^{b-1} \leq \lambda$. The inequality $v_2(a_{n,m,r})\geq 2m+\alpha r$ for $r\geq r_\alpha$ then follows by computing a lower bound on the Newton function of $q(u)$ partitioning the Newton slopes up to $r_\alpha$ in the intervals $[0,\alpha]$, $[\alpha,(1+\alpha)/2]$ and $[(1+\alpha)/2,\infty)$.

Fix $\alpha<1$. The fact that $v_2(a_{n,m,r}/m!)\geq m$ for all $r$ and $\geq m+\alpha r$ for $r\geq r_\alpha$ implies that up to precision $M$:
\begin{align*}
\Tr(T_{2^n}|S_{2k})&\equiv -1-\sum_{\substack{r_\alpha\leq r<(M-m)/\alpha\\ \textrm{ or }  r<r_\alpha,m<M}}\frac{a_{n,m,r}k^m}{m!}\sum_{0\leq u<2^{n/2}}H(2^{n+2}-(1+2u)^2)u^r\pmod{2^M}.
\end{align*}
It is therefore necessary and sufficient to compute the sums
\[\mathscr{H}^{\odd}_{r}(2^n) = \sum_{1\leq t<2^{n/2+1}, 2\nmid t}H(2^{n+2}-t^2)t^r\]
for $r\leq M/\alpha$. Such sums have long been studied in
connection with Fourier coefficients of modular forms \cite{lygeros-rozier,hurwitz-congruence}.

In the remainder of this section we will explain how to make the
computation of $\mathscr{H}^{\odd}_r(2^n)$ faster. We begin with a result that states that it suffices to compute $$\mathscr{H}_r(2^n)=\displaystyle \sum_{1\leq t<2^{n/2+1}}H(2^{n+2}-t^2)t^r.$$
 
\begin{lemma}\label{l:hurwitz odd sums}
We have
\begin{align*}
\mathscr{H}_r^{\odd}(2^{2m})
  &=\mathscr{H}_r(2^{2m})-2^{r+1}\mathscr{H}_r(2^{2(m-1)})-2^{rm+1}/3\\
\mathscr{H}_r^{\odd}(2^{2m-1})
  &=\mathscr{H}_r(2^{2m-1})-2^{r+1}\mathscr{H}_r(2^{2(m-1)-1})-2^{rm-1}.
\end{align*}
\end{lemma}
\begin{proof}
Collecting terms by 2-adic valuation we see that
\begin{align*}
\mathscr{H}_r(2^{n})&=\sum_{0\leq 2e<n+2} 2^{er}\sum_{1\leq t < 2^{n/2+1-e},2\nmid t}H(2^{n+2}-2^{2e}t^2)t^r.
\end{align*}
If $n+2-2e<3$ then either $n=2e-1$ in which case the inner sum
is $H(2^{2e})=2^{e-1}-1/2$ or $n=2e$ in which case the inner sum
is $H(3\cdot 2^{2e})=2^e-2/3$.

Let us turn to the general case $n+2-2e\geq 3$. For any $m$ the
function $\varphi_m$ that appears in the definition of Hurwitz
class numbers is multiplicative, and therefore so is the divisor
sum $\displaystyle \sum_{a^2\mid m/d_m}\varphi_m(a)$. For odd $t$ we have $2^{n+2-2e}-t^2\equiv 7\pmod{8}$ and therefore $w=1$, and we compute $$\displaystyle 
H(2^{2e}(2^{n+2-2e}-t^2))=H(2^{n+2-2e}-t^2)\sum_{k=0}^e\varphi_{2^{n+2-2e}-t^2}(2^k)=2^e  H(2^{n+2-2e}-t^2).$$ We deduce that
\begin{align*}
\mathscr{H}_r(2^{2m})&=\sum_{e=0}^{m-1}2^{(r+1)e} \mathscr{H}_r^{\odd}(2^{2(m-e)})+2^{(r+1)m}-2^{rm+1}/3\\
\mathscr{H}_r(2^{2m-1})&=\sum_{e=0}^{m-1}2^{(r+1)e} \mathscr{H}_r^{\odd}(2^{2(m-e)-1})+2^{(r+1)m-1}-2^{rm-1}.
\end{align*}
Solving for the odd Hurwitz sums we conclude the desired formulas.

\end{proof}

We now turn to computing the sums $\mathscr{H}_r(2^n)$. The
following procedure is based on \cite[\S4]{lygeros-rozier}. We
will denote $f(n)\approx g(n)$ if $f(n)-g(n)$ can be computed in
time which is polynomial in $n$. Then the Eichler-Selberg trace
formula implies that
\[\Tr(T_{2^n}|S_{2k})\approx
\sum_{i=0}^{k-1}(-1)^i\binom{2k-2-i}{i}2^{ni}\mathscr{H}_{2k-2-2i}(2^n).\]
In \S\ref{s:hecke} we explained how to do compute the LHS by computing $\Tr(T_2^n|S_{2k})$, and this can be evaluated explicitly as the sum of the $n$-th powers of the normalized $a_2$ coefficients on a basis of $S_{2k}$, which can be done in polynomial time. Solving the system of equations we conclude that $\mathscr{H}_r(2^n)$ can be computed in polynomial time for every even $r$.

The computation of $\mathscr{H}_r(2^n)$ for odd $r$ would
similarly be related to $\Tr(T_{2^n}|S_k(N, \varepsilon))$ for
odd weights $k$ and (necessarily) odd characters
$\varepsilon:(\mathbb{Z}/N \mathbb{Z})^\times\to
\mathbb{C}^\times$. For simplicity suppose $N=\ell \equiv
7\pmod{8}$ is a prime number and $\varepsilon
=\legendre{\cdot}{\ell}$, an odd character with $\varepsilon(2)=1$. The Eichler-Selberg
trace formula \cite[p. 370]{knightly-li:traces} in this case implies that
\begin{align*}
-\Tr(T_{2^n}|S_k(\ell, \varepsilon))
&\approx
           \sum_{i<k/2-1}(-1)^i\binom{2k-2-i}{i}2^{ni}\left(\sum_{\ell\mid 2^{n+2}-t^2}\legendre{t}{\ell}H(2^{n+2}-t^2)t^{k-2-2i}+\right.\\
  &\left.+\ell\sum_{\ell^2\mid 2^{n+2}-t^2}H\left(\frac{2^{n+2}-t^2}{\ell^2}\right)t^{k-2-2i}+2\sum \legendre{\rho_t}{\ell} H(2^{n+2}-t^2)t^{k-2-2i}\right),
\end{align*}
where $t$ varies such that $1\leq t<2^{n/2+1}$ and the last sum
is over $t$ such that $\legendre{t^2-2^{n+2}}{\ell}=1$, i.e.,
$\rho_t\in \mathbb{F}_\ell$.

As in the even weight case we may compute $\Tr(T_{2^n}|S_k(\ell,\varepsilon))$ for any given odd $k$ in polynomial time. Solving the system of equations above, we may compute a Hurwitz sum $\displaystyle \sum H(2^{n+2}-t^2)t^r$ for odd $r$ over $t$ such that $\legendre{t^2-2^{n+2}}{\ell}=1$ and $\legendre{\rho_t}{\ell}=1$ in terms of the sums over $t$ such that $\ell\mid 2^{n+2}-t^2$ and such that $\legendre{t^2-2^{n+2}}{\ell}=1$ and $\legendre{\rho_t}{\ell}=-1$. Asymptotically, this eliminates the need to compute a quarter of the terms in $\mathscr{H}_r(2^n)$.

\section{The Buzzard-Calegari Conjecture}\label{s:bc}
In this section we turn our attention to the Buzzard-Calegari
Conjecture and the proof of Theorem \ref{MainTh}.

\medskip

Let $2k\geq 12$ and $m=\dim S_{2k}$ the dimension of the space of cusp forms of weight $2k$ and level $\SL_2(\mathbb{Z})$. It is elementary to show that the Newton polygon of the Buzzard-Calegari polynomial 
\[P_{\BC}(X)=1+\sum_{n=1}^m X^n\prod_{j=1}^n
\frac{2^{2j}(2k-8j)!(2k-8j-3)!(2k-12j-2)}{(2k-12j)!(2k-6j-2)!}\]
is the same as the Newton polygon of the points $\{(n, a_n)|0\leq n\leq m\}$ where the nonnegative integers $a_n$ can be described explicitly as follows. For an integer $\ell$ we denote $\delta_\ell=1$ if $\ell\equiv 1\pmod{6}$ and $0$ otherwise. Then
\begin{equation}\label{eq:bc}a_n = \frac{3n(n+1)}{2}+\sum_{\ell=3n+4}^{6n+1}e_{n,\ell}
v_2(k-\ell),
\end{equation}
where
\[e_{n,\ell}=\begin{cases}\left\lfloor\frac{\ell-1}{3}\right\rfloor -n&3n+4\leq
\ell<4n\\
\delta_\ell
+n-1-\left\lfloor\frac{\ell}{6}\right\rfloor&\ell=4n,4n+1\\
\delta_\ell
+n-\left\lfloor\frac{\ell}{6}\right\rfloor&4n+2\leq\ell\leq 6n\\
1&\ell=6n+1\end{cases}.\]

The Buzzard-Calegari conjecture posits that the Newton polygons
of points with heights prescribed by equations \eqref{eq:np} (Corollary \ref{c:wedge n}) and
\eqref{eq:bc} are the same. We will experimentally verify our truncated version of this conjecture  by
making explicit the list of transcendental numbers $\Omega_i$
and integers $n_j,d_j,u_j$ in Corollary \ref{c:wedge n}.

\begin{theorem}\label{t:main}
For $1\leq n\leq 15$ there exist $\Omega_i\in \mathbb{Z}_2$ and
integers $n_j, u_j$ such that:
\[v_2(\Tr(\wedge^n T_2|S_{2k}))=\frac{3n(n+1)}{2} +\sum_i
v_2(k-\Omega_{i})+\sum_j\min(n_{j},2 v_2(k-u_{j})),\]
for all large enough $k$. The number of expressions of the form $v_2(k-\alpha)$, counted with multiplicity, is the same as in  \eqref{eq:bc}. \end{theorem}
\begin{proof}
We implemented Corollary \ref{c:wedge n} in Sage \cite{sage}, the
results being tabulated in the Appendix.
In the case of $n=1$, the one constant $\Omega \in \mathbb{Z}_2$ appearing in the expression above can be computed using the $2$-adic log map:
\[\Omega = \frac{1}{2}\left(\frac{\log_2(1-2
 \rho)}{\log_2(1-\rho)}+1\right),\] where
$\rho=\frac{1-\sqrt{-7}}{2}$ and therefore for all $k\geq 2$ we
have $$v_2(\Tr (T_2|S_{2k}))=3+v_2(k-\Omega).$$
Indeed, since the exponential sum contains only unit bases we are not constrained to $k\gg 0$.

For the convenience of the reader, we explain in more detail the
case $n=4$. In this case, the power series $f_0(k)$ obtained
from $\Tr(\wedge^4 T_{2}|S_{2k})$ has content $30=\frac{3\cdot 4\cdot 5}{2}$ and factors,
up to an invertible power series and precision $10$, as
\[(k-789)(k-23)(k-337)(k-25)(k-18)(k-22)(k-980)(k-720)(k^2-38k+361).\]
The linear factors provide the list of $\Omega_j$. Applying Lemma \ref{l:vp poly} to the unique quadratic polynomial $q(k)$ we obtain $\lambda=0$, $\delta=5$ and $(u_n)=\{1,1,0,0,1\}$ and therefore $v_2(q(k))= \min(15, 2 v_2(k - 19))$, thereby obtaining the desired explicit formula for $v_2(\Tr(\wedge^4 T_2|S_{2k}))$.
\end{proof}

\begin{remark}
It is important, for applications, to know to what precision one
can compute the constants in Theorem \ref{t:main}. Suppose one
works in $\mathbb{Q}_2$ with $M$ digits of precision and, in
Theorem \ref{t:vp seq}, one approximates power series with
polynomials of degree $D\leq M$. If $f(x) =
\sum\limits_{u=1}^\infty x^u/u! \sum \limits_{n=1}^m a_n
(\log_2(b_n^2))^u$ is approximated with $f_D(x) = \sum
\limits_{u=1}^D x^u/u! \sum \limits_{n=1}^m a_n
(\log_2(b_n^2))^u$ then $f(x)\equiv f_D(x)\pmod{2^D}$. Indeed,
since $v_2(\log_2(b_n^2))\geq 2$ it follows that $$v_2(1/u! \sum
\limits_{n=1}^m a_n (\log_2(b_n^2))^u)\geq 2u-v_2(u!)\geq u$$ for
all $u$.

Removing the content $\frac{3n(n+1)}{2}$ of the power series
$f_D(x)$, we may compute the nonnegative slope polynomials
$P_D(x)$ with precision $P=D-\frac{3n(n+1)}{2}$. As long as each
$n_j$ in Theorem \ref{t:main} is smaller than this precision we
know each $\Omega_i$ to precision $P$.

The tables in the Appendix were obtained using $N=10000$ and $D=500$ which means that we identified each constant for $\wedge^n T_2$ up to precision $500-\frac{3n(n+1)}{2}$, so overall precision at least 140.
\end{remark}
\begin{remark}
It is computationally useful in the proof of Theorem \ref{c:main} that the explicit constants $\Omega_j$ (resp. $u_j$) in Theorem \ref{t:main} appear ``2-adically close'' to corresponding $\ell$ in \eqref{eq:bc}. By this we mean the following: the total number of constants $\Omega_j$ (counted with multiplicity 1) and $u_j$ (counted with multiplicity 2) in Theorem \ref{t:main} equals the total number of $\ell$ (counted with multiplicity $e_{n,\ell}$) in \eqref{eq:bc}. Moreover, for each $\Omega_j$ (resp. $u_j$) there exists an $\ell$ such that $v_2(\Omega_j-\ell)$ (resp. $v_2(u_j-\ell)$) is large, though we are not able to quantify this valuation in general. The pairing of $\Omega_j$ (resp. $u_j$) and the corresponding $\ell$ appears in the Appendix.
\end{remark}

Our main computational result is a partial verification of the
Buzzard-Calegari conjecture using Theorem \ref{t:main}. Recall that if $N$ is a Newton polygon by the truncation at $m$ we mean the portion $N^{\leq m}$ of $N$ in the region $0\leq n\leq m$. Also, $N_1\geq N_2$ means that the Newton polygon $N_1$ lies on or above the Newton polygon $N_2$.

\begin{theorem}\label{c:main}
For $k\gg 0$ the Newton polygon
$N_{\BC}$ of the Buzzard-Calegari polynomial $P_{\BC}$
has some vertex $n$ in the interval $[7,15]$. For any such vertex:
\[N_{\BC}^{\leq n}=N(P_{T_2}^{\deg\leq n}).\]
Moreover, $N(P_{T_2}^{\deg\leq 15})\geq N_{\BC}$.
\end{theorem}

We shall need the following technical result. 

\begin{lemma}\label{l:vertex}
Suppose $k>6$ is an integer and $n\geq 1$. The Newton polygon $N_{\BC}$ has a vertex at $n$ if and only if the Newton polygon of $P_{\BC}(X)$ truncated in degree $n+\frac{2}{3}\sum\limits_{\ell=3n+4}^{6n+1}e_{n,\ell} v_2(k-\ell)$ does as well, in which case $N_{\BC}^{\leq n}=N(P_{\BC}^{\deg\leq n})$.
\end{lemma}
\begin{proof}
The ``only if'' direction is straightforward. To check that
$N_{\BC}$ has a vertex at $n$ one needs to check that
slopes in the region $\leq n$ are $<$ the slopes in the region $\geq n$, i.e., for
all vertices $a<n$ and vertices $b>n$:
\[\frac{1}{n-a}\left(\frac{3n(n+1)}{2}+\sum
e_{n,\ell}v_2(k-\ell)-\frac{3a(a+1)}{2}-\sum
e_{a,\ell}v_2(k-\ell)\right) \]\[<\frac{1}{b-n}\left(\frac{3b(b+1)}{2}+\sum
e_{b,\ell}v_2(k-\ell)-\frac{3n(n+1)}{2}-\sum
e_{n,\ell}v_2(k-\ell)\right).\]
This inequality is automatic for all $b> n+\frac{2}{3}\sum\limits_{\ell=3n+4}^{6n+1}e_{n,\ell} v_2(k-\ell)$ and therefore one only needs to compute explicitly slopes in the Newton polygon truncated at this degree.

\end{proof}

\begin{proof}[{\bf Proof of Theorem \ref{c:main}}]
We now describe our computational approach to verifying the corollary. At each iteration $i$ we are left with verifying the case of $k$ varying in a collection of 2-adic balls $\mathcal{C}_i=\{\mathcal{B}(a_{i,j}, n_{i,j})\}$ (where $\mathcal{B}(a,m)=\{x\in \mathbb{Z}_2\mid v_2(x-a)>m\}$), initially starting with $\mathcal{C}_0=\{\mathbb{Z}_2\}$.

We begin with the following observation:
\begin{enumerate}
\item if $v_2(k-\Omega)\neq
v_2(k-\ell)$ then either $k\in \mathcal{B}(\Omega,
v_2(\Omega-\ell))$ or $k\in \mathcal{B}(\ell, v_2(\Omega-\ell))$;
\item if $n>v_2(u-\ell)$ and $\min(n,
v_2(k-u))\neq v_2(k-\ell)$ then again either $k\in \mathcal{B}(u,
v_2(u-\ell))$ or $k\in \mathcal{B}(\ell, v_2(u-\ell))$;
\item finally, if $u=\ell$ then
$\min(n, 2\cdot v_2(k-\ell))\neq 2\cdot v_2(k-\ell)$ then $k\in \mathcal{B}(\ell, n/2)$.
\end{enumerate}
Therefore the two sets of vertices $\{(n, v_2(\Tr(\wedge^n
T_2|S_{2k})))\mid n\leq 15\}$ and $\{(n,a_n)\mid n\leq 15\}$
coincide unless for some $n$ and a pair $(\Omega, \ell)$
(resp. triple $(m,u,\ell)$) in the Appendix we have
$v_2(k-\Omega)\neq v_2(k-\ell)$ (resp. $\min(m,2\cdot v_2(k-u))\neq
2\cdot v_2(k-\ell)$). By the discussion above we conclude that we next need to treat the case when $k$ varies in one of the $2$-adic balls in the collection:
\[\mathcal{C}_1 = \{\mathcal{B}(\Omega,
v_2(\Omega-\ell))\}\cup \{\mathcal{B}(\ell,
v_2(\Omega-\ell))\}\cup \{\mathcal{B}(u,
v_2(u-\ell))\}\cup \{\mathcal{B}(\ell,
v_2(u-\ell))\}\cup \{\mathcal{B}(\ell,
n_j/2)\}.\]
Suppose we are in iteration $i$ and we desire to verify
Theorem \ref{c:main} for all sufficiently large $k$ in
$\mathcal{B}(a, m)\in \mathcal{C}_i$.

For an integer $N$ we denote:
\begin{align*}
\mathcal{S}_{T_2}(\mathcal{B}(a,m), N) &= \{(n, \frac{3n(n+1)}{2}+\sum v_2(a-\Omega_i)+\sum \min(n_j, 2\cdot v_2(a-u_j))\mid 0\leq n\leq N\}\\
\mathcal{S}_{\BC}(\mathcal{B}(a,m), N) &= \{(n,
                                      \frac{3n(n+1)}{2}+\sum
                       e_{n,\ell}\cdot
                                      v_2(a-\ell)u_j))\mid 0\leq
                                      n\leq N\}.
\end{align*}

Note that if $k\in \mathcal{B}(a,m)$ then $v_2(k-\Omega)\geq
v_2(a-\Omega)$ if $v_2(a-\Omega)\leq m$ with equality if the latter inequality is strict; in this case we say that $v_2(k-\Omega)$ is \emph{precisely computed} in the ball $\mathcal{B}(a,m)$. We conclude that 
the points in $\mathcal{S}_{T_2}$ lie below the points $\{(n, v_2(\Tr \wedge^n T_2|S_{2k}))\mid n\leq N\}$, and similarly for $\mathcal{S}_{\BC}$, and the same can be said of their Newton polygons. We say that the $n$-th point in $\mathcal{S}_{T_2}$ or $\mathcal{S}_{\BC}$ is precisely computed in the ball $\mathcal{B}(a,m)$ if every valuation in the formula for $n$ is precisely computed in $\mathcal{B}(a,m)$. Note that if every vertex of the Newton polygon of $\mathcal{S}_{T_2}$ is precisely computed then this Newton polygon is the actual Newton polygon of the points $\{(n, v_2(\Tr \wedge^n T_2|S_{2k}))\mid n\leq N\}$, and similarly for $\mathcal{S}_{\BC}$.

We begin with determining the largest $n\leq 15$ for which $n$
is a vertex of $N_{\BC}$. First, if $n$ is a vertex of $N_{\BC}$
then it must also be a vertex of the Newton polygon $N_{\BC, N}$
of the points $\{(n,a_n)\mid 0\leq n\leq N\}$ for every $N$. On
a first round, to check if $n$ is a vertex of $N_{\BC}$ we first
verify if $n$ is a vertex of $N_{\BC, 2n+1}$. The reason for
this choice is the following: if $a_n$ is not precisely computed
in the ball $\mathcal{B}(a,m)$ then $a$ is close to some $\ell$
appearing in the expression for $a_n$. However, the set of
$\ell$-s appearing in the formula for $a_{2n+1}$ is disjoint
from the sets of $\ell$-s appearing in $a_n$ and therefore the
last vertex in $N_{\BC, 2n+1}$ is likely to be precisely
computed. As a proxy for $N_{\BC, N}$ we use the Newton polygon
of $\mathcal{S}_{\BC}(\mathcal{B}(a,m), N)$. If $a_n$ is not a
vertex of $\mathcal{S}_{\BC}(\mathcal{B}(a,m), 2n+1)$ and the
Newton polygon is precisely computed (in the sense mentioned
above) then $n$ is definitely not a vertex of $N_{\BC}$ and we
discard it. If $a_n$ is precisely computed and is a vertex of
$\mathcal{S}_{\BC}(\mathcal{B}(a,m), 2n+1)$ (not necessarily
precisely computed) then $n$ is definitely a vertex of $N_{\BC,
  2n+1}$, and plausibly a vertex of $N_{\BC}$. In all other
cases the computation is imprecise and we add $\mathcal{B}(a,m)$
to the set $\mathcal{CI}_i$ of balls $\mathcal{B}(a,m)$ where
the computation was not precise. It remains to further verify if
$n$ is a vertex of $N_{\BC}$ in the plausible case. For this, we
apply Lemma \ref{l:vertex}. Since $a_n$ was precisely computed
so is the upper bound $N$ from Lemma \ref{l:vertex}. We apply to
$N_{\BC, N}$ the procedure described above in the case of
$N_{\BC, 2n+1}$. Either $n$ is definitely a vertex of $N_{\BC,
  N}$ and, by Lemma \ref{l:vertex}, definitely a vertex of
$N_{\BC}$, or $n$ is definitely not a vertex, or
$\mathcal{B}(a,m)$ is added to $\mathcal{CI}_i$. We determine
the largest $n\leq 15$ which is definitely a vertex by going
backwards from $15$ with this verification until we arrive at a
vertex.

If $\mathcal{B}(a,m)$ has not yet been added to
$\mathcal{CI}_i$, we compute the Newton polygons of
$\mathcal{S}_{T_2}(\mathcal{B}(a,m), n)$ and
$\mathcal{S}_{\BC}(\mathcal{B}(a,m), n)$. If these Newton
polygons are precisely computed we verify the Buzzard-Calegari
conjecture up to the vertex $n$ by checking if the polygons are equal.

At the end of iteration $i$ we are left with a subset $\mathcal{CI}_i\subset \mathcal{C}_i$ of balls over which our computations were not precise enough to verify Buzzard-Calegari. For each $\mathcal{B}(a,m)\in \mathcal{CI}_i$ we produce two half balls $\mathcal{B}(a, m+1)$ and $\mathcal{B}(a+2^{m+1}, m+1)$ and add them to $\mathcal{C}_{i+1}$, the collections $\mathcal{C}_{i+1}$ and $\mathcal{CI}_i$ having the same union. In iteration $i+1$ we will verify the corollary for the balls in $\mathcal{C}_{i+1}$, each of which has smaller radius than the balls in $\mathcal{C}_i$, and therefore more likely to yield a precise computation of Newton polygons. 

In 23 iterations which ran for about 6 minutes on a laptop the algorithm
examined every one of the 1058 balls in $\mathcal{C}_1$ and
established the validity of the theorem, including the existence of a vertex of $N_{\BC}$ in the interval $[7,15]$.

Finally, we remark that the constants in the Appendix are known to sufficient precision to make the above computations correct. Indeed, by the remark after Theorem \ref{t:main} we know these constants up to precision 140. Since the balls $\mathcal{B}(a,m)$ in every $\mathcal{C}_i$ have radius valuation at most $m=84$, it follows that if the center $a$ is approximated by $a'$ with precision 140 then $\mathcal{B}(a, m)=\mathcal{B}(a' ,m)$.
\end{proof}

We now turn our attention to the relationship between
$N(P_{T_2})$ and $N(P_{T_2}^{\deg \leq m})$. The following result is a variant of Lemma \ref{l:vertex} for the characteristic polynomial of $T_2$.

\begin{lemma}\label{l:vertex T2}
Suppose $k\gg 6$ is an integer and $1\leq n\leq 15$. The Newton polygon of $P_{T_2}(X)$ has a vertex at $n$ if and only if so does the Newton polygon of $P_{T_2}(X)$ truncated in degree $$m_{n,k}=\frac{1}{4}(8+\lambda+\sqrt{(8+\lambda)^2-8(4+\lambda n-f_n)}),$$ where $f_n$ is the expression on the right hand side of Theorem \ref{t:main} and $\lambda=\max\limits_{0\leq a<n}\frac{f_n-f_a}{n-a}$. In this case, we further have $N(P_{T_2})^{\leq n}=N(P_{T_2}^{\deg\leq n})$.
\end{lemma}

\begin{proof}


Let $b_n$ be the coefficient of $X^n$ in the characteristic polynomial of $T_p$ acting on the space of modular forms $M_{k+jd(p-1)}$. Set $d_j=\dim M_{k+j(p-1)}$, $d=4$ if $p=2$, $3$ if $p=3$ and $1$ if $p\geq 5$. Then \cite[Lemma 2.1]{vonk} used as in the proof of \cite[Lemma 3.1]{wan} implies that 
\[v_p(b_n(k+jd(p-1)))\geq \frac{d(p-1)}{p+1}\left(\sum_{u=0}^\ell u
(d_u-d_{u-1})+(\ell+1)(n-d_\ell)\right)-\left(1+\frac{d-1}{p+1}\right)\]
where $\ell$ is such that $d_\ell\leq n<d_{\ell+1}$.

Note that $d_u=\lfloor \frac{k+ud(p-1)}{12}\rfloor+\delta_u$,
where $\delta_u=1$ unless $k+ud(p-1)\equiv 2\pmod{12}$ in which
case $\delta_u=0$. This implies that $\frac{k+ud(p-1)}{12}-1\leq d_u\leq \frac{k+ud(p-1)}{12}+1$ and \[\frac{12n-12-k}{d(p-1)}-1<\ell \leq \frac{12n+12-k}{d(p-1)}.\]

Therefore
\begin{align*}
v_p(b_n(k+jd(p-1)))&\geq
                    \frac{d(p-1)}{p+1}\left((\ell+1)n-\sum_{u=0}d_u\right)-\left(1+\frac{d-1}{p+1}\right)n\\
  &\geq \frac{d(p-1)}{p+1}\left((\ell+1) \left(n-\frac{k}{12}-1\right)-\frac{d(p-1)\ell(\ell+1)}{24}\right)-\left(1+\frac{d-1}{p+1}\right)n.
\end{align*}
Using the bounds on $\ell$ we see that
\begin{align*}
v_p(b_n(k+jd(p-1)))&\geq
                    \frac{d(p-1)}{p+1}(\ell+1)\left(n-\frac{k}{12}-1-\frac{d(p-1)\ell}{24}\right)-\left(1+\frac{d-1}{p+1}\right)n\\
                  &>\frac{d(p-1)}{p+1}\frac{12n-12-k}{d(p-1)}(n/2-k/24-3/2)-n\\
  &=\frac{6}{p+1}(n-k/12-1)(n-k/12-3/2)-\left(1+\frac{d-1}{p+1}\right)n.
\end{align*}
When $p=2$, $d=4$ and $k$ can be taken between $0$ and $3$. We
therefore obtain the uniform bound
\begin{equation}\label{eq:vonk}
v_2(b_n)\geq 2n^2-8n+35/8>2(n-2)^2-4.
\end{equation}
for all $n\geq 1$ and $k\gg 6$. 

As in Lemma \ref{l:vertex}, to verify whether $n$ is a vertex of $N(P_{T_2})$ it suffices to compare slopes (using Theorem \ref{t:main} for $n\leq 15$): 
\[\max_{0\leq a<n}\frac{f_n-f_a}{n-a}<\frac{v_2(b_m)-f_n}{m-n}\]
for all $m>n$. However, using \eqref{eq:vonk} we see that the above inequality is automatically satisfied for $m\geq m_{n,k}$ and therefore it suffices to verify that $n$ is a vertex of the Newton polygon of the truncation in degree $m_{n,k}$.
\end{proof}
\begin{remark}
In principle, one should be able to strengthen Theorem
\ref{c:main} using Lemma \ref{l:vertex T2} and obtain a result
of the form: the Buzzard-Calegari Conjecture \ref{c:bc} is true
up to $n\leq D$. As stated, our computations suffice to show
\[N_{\BC}^{\leq 4}=N(P_{T_2})^{\leq 4},\]
for all sufficiently large weights. Unfortunately, to verify Theorem \ref{c:main} with $N(P_{T_2})^{\leq n}$ instead of $N(P_{T_2}^{\deg\leq n})$ one would need to compute Theorem \ref{t:main} up to $n=42$, which is currently unfeasible. However, computing Theorem \ref{t:main} up to $n=16$ would suffice to show $N_{\BC}^{\leq 5}=N(P_{T_2})^{\leq 5}$.

Finally, for each 2-adic ball $\mathcal{B}(a,m)$ from the proof
of Theorem \ref{c:main} we have computed a $D$ such that $N_{\BC}^{\leq D}=N(P_{T_2})^{\leq D}$, obtaining the example that $N_{\BC}^{\leq 13}=N(P_{T_2})^{\leq 13}$ for $k\equiv 34\pmod{2^{11}}$. 
\end{remark}

We end this section with an appealing application of Theorem
\ref{c:main} to Hatada's congruences, namely Corollary \ref{cor}. The computations that led to Theorem \ref{c:main}
imply that the lowest two Newton slopes of $P_{T_2}^{\deg \leq 15}$ are precisely the pairs listed in the introduction, the values depending on $k\Mod {64}$ as specified. By \eqref{eq:vonk} all Newton slopes of
$P_{T_2}$ in the region $\geq 15$ are at least $12$, and therefore
Theorem \ref{c:main} suffices to verify that the lowest 2 slopes of $P_{T_2}$ coincide with the lowest two slopes of $P_{T_2}^{\deg \leq 15}$.

\newpage

\section{Appendix: 2-adic valuations of the traces of \texorpdfstring{$T_2$}{}}

In the following table, for each $n\leq 5$ we make explicit the terms that appear in $v_2(\Tr(\wedge^n T_p|S_{2k}))$ from Theorem~\ref{t:main}, and those defining the sequence $a_n$ from equation (\ref{eq:bc}) in Section~\ref{s:bc}. The 2-adic numbers $\Omega_i$ are given modulo $2^{50}$. 

\smallskip

The complete table for all $n\leq 15$ is available online \cite{data}. 

\begin{longtable}{l|l|l|l}
  $n$ & $v_2(k-\Omega_i)$ or $\min(n_j, 2 v_2(k-u_j))$ in $v_2(\operatorname{Tr}(\wedge^n T_2|S_{2k}))$ & $v_2(k-\ell)$ in $a_n$ & $v_2(\Omega_i-\ell)$ or  $v_2(u_j-\ell)$ \\
  \hline
  $1$ & $v_2(k- 442980431217671)$ & $v_2(k-7)$ & $10$\\
  \hline
  $2$  & $v_2(k-791247700865546)$ & $v_2(k-10)$ & $9$ \\
 & $v_2(k-31828396041227)$ & $v_2(k-11)$ & $10$ \\
 & $v_2(k-335062469580877)$ & $v_2(k-13)$ & $6$ \\
  \hline
$3$    & $v_2(k-48255093739981)$ & $v_2(k-13)$ & $6$ \\
 & $v_2(k-895017375933454)$ & $v_2(k-14)$ & $10$ \\
 & $v_2(k-16843008520207)$ & $v_2(k-15)$ & $12$ \\
 & $v_2(k-250702637217616)$ & $v_2(k-16)$ & $6$ \\
 & $v_2(k-46624142875857)$ & $v_2(k-17)$ & $6$ \\
 & $v_2(k-474794944364563)$ & $v_2(k-19)$ & $28$ \\
  \hline
  $4$
 & $v_2(k-798532487856848)$ & $v_2(k-16)$ & $6$ \\
 & $v_2(k-658899949170001)$ & $v_2(k-17)$ & $6$ \\
 & $v_2(k-568752135614482)$ & $v_2(k-18)$ & $12$ \\
 & $\min(15, 2 \cdot v_2(k - 19))$ & $2\cdot v_2(k-19)$ & $\infty$ \\ 
 & $v_2(k-1103383114654676)$ & $v_2(k-20)$ & $6$ \\
 & $v_2(k-60661288646421)$ & $v_2(k-21)$ & $8$ \\
 & $v_2(k-1080512839942166)$ & $v_2(k-22)$ & $31$ \\
 & $v_2(k-339362545926167)$ & $v_2(k-23)$ & $33$ \\
 & $v_2(k-824086375843865)$ & $v_2(k-25)$ & $11$ \\
  \hline
  $5$
 & $v_2(k-912948839579667)$ & $v_2(k-19)$ & $19$ \\
 & $v_2(k-929666093061716)$ & $v_2(k-20)$ & $6$ \\
 & $v_2(k-1090275108829461)$ & $v_2(k-21)$ & $8$ \\
 & $\min(15, 2 \cdot v_2(k - 22))$ & $2\cdot v_2(k-22)$ & $\infty$ \\ 
 & $\min(16, 2 \cdot v_2(k - 151))$ & $2\cdot v_2(k-23)$ & $7$ \\ 
 & $v_2(k-215022683507480)$ & $v_2(k-24)$ & $8$ \\
 & $v_2(k-188349340154137)$ & $v_2(k-25)$ & $8$ \\
 & $v_2(k-25411498307609)$ & $v_2(k-25)$ & $12$ \\
 & $v_2(k-255292856074266)$ & $v_2(k-26)$ & $36$ \\
 & $v_2(k-893284478091291)$ & $v_2(k-27)$ & $36$ \\
 & $v_2(k-378319707637788)$ & $v_2(k-28)$ & $11$ \\
 & $v_2(k-436532622338077)$ & $v_2(k-29)$ & $11$ \\
      & $v_2(k-669602715533343)$ & $v_2(k-31)$ & $27$\\
\end{longtable}


\bibliographystyle{spbasic} 

\end{document}